\documentclass[11pt,review]{article}
\usepackage{amsmath,amsthm,amssymb,amscd}
\usepackage{graphicx}
\usepackage{graphics}
\usepackage{verbatim}

\usepackage{mathrsfs}
\usepackage{color}
\usepackage[top=1in, bottom=1in, left=1.25in, right=1.25in]{geometry}
\usepackage{enumerate}

\newtheorem{theorem}{Theorem}[section]
\newtheorem{proposition}[theorem]{Proposition}

\newtheorem{lemma}[theorem]{Lemma}

\newtheorem{corollary}[theorem]{Corollary}
\newtheorem{claim}{Claim}[theorem]

\begin{document}

\title{Choosability in signed planar graphs}

\vspace{3cm}

\author{Ligang Jin\thanks{supported by Deutsche Forschungsgemeinschaft (DFG) grant STE 792/2-1}, Yingli Kang\thanks{Fellow of the International Graduate School ``Dynamic Intelligent Systems''}, Eckhard Steffen\thanks{
		Paderborn Institute for Advanced Studies in
		Computer Science and Engineering,
		Paderborn University,
		Warburger Str. 100,
		33098 Paderborn,
		Germany;			
		ligang@mail.upb.de, yingli@mail.upb.de, es@upb.de}}

\date{}

\maketitle

\begin{abstract}
This paper studies the choosability of signed planar graphs. We prove that every signed planar graph is 5-choosable and that there is a signed planar graph which is not 4-choosable while the unsigned graph is 4-choosable. For each $k \in \{3,4,5,6\}$, every signed planar graph without circuits of length $k$ is 4-choosable. Furthermore, every signed planar
graph without circuits of length 3 and of length 4 is 3-choosable.
We construct a signed planar graph with girth 4 which is
not 3-choosable but the unsigned graph is 3-choosable.
\end{abstract}

\section{Introduction}
This paper discusses simple graphs.
Let $G$ be a graph with vertex-set $V(G)$ and edge-set $E(G)$.
We say a vertex $u$ is a \emph{neighbor} of another vertex $v$ if $uv\in E(G)$.
If $v\in V(G)$, then $d(v)$ denotes the degree of $v$ and furthermore, $v$ is called a \emph{$k$-vertex} (or \emph{$k^+$-vertex} or \emph{$k^-$-vertex}) if $d(v)=k$ (or $d(v)\geq k$ or $d(v)\leq k$).
Similarly, a \emph{$k$-circuit} (or \emph{$k^+$-circuit} or \emph{$k^-$-circuit}) is a circuit of length $k$ (or at least $k$ or at most $k$), and if $G$ is planar then a \emph{$k$-face} (or \emph{$k^+$-face} or \emph{$k^-$-face}) is a face of size $k$ (or at least $k$ or at most $k$).
Let $[x_1\ldots x_k]$ denote a $k$-circuit with vertices $x_1,\ldots,x_k$ in cyclic order.
If $X\subseteq V(G)$, then $G[X]$ denotes the subgraph of $G$ induced by $X$, and $\partial(X)$ denotes the set of edges between $X$ and $V(G)\setminus X$.

Let $G$ be a graph and $\sigma: E(G)\rightarrow \{1,-1\}$ be a mapping. The pair $(G,\sigma)$ is called a \emph{signed graph}, and $\sigma$ is called a \emph{signature} of $G$. An edge $e$ is \emph{positive} (or \emph{negative}) if $\sigma(e)=1$ (or $\sigma(e)=-1$).
Denote by $(G,+)$ the signed graph $(G,\sigma)$ with $\sigma(e)=1$ for each $e\in E(G)$.
A graph with no signature is usually called an \emph{unsigned graph}.
A circuit of a signed graph is balanced (unbalanced) if it contains an even (odd) number of negative edges.

Zaslavsky \cite{Zaslavsky_1982} defines a (signed) coloring of a signed graph $(G,\sigma)$ with $k$ colors or with $2k+1$ signed colors to be a mapping
$c : V(G) \longrightarrow \{-k, -(k-1), \dots, -1,0,1 \dots, (k-1), k\}$ such that for every edge $uv$ of $G$, $c(u)\neq c(v)$ if $\sigma(uv)=1$, and $c(u)\neq -c(v)$ if $\sigma(uv)=-1$. 
Recently, M\'{a}\v{c}ajov\'{a}, Raspaud and \v{S}koviera \cite{Raspaud_2014} introduced a \emph{$k$-coloring} of $(G,\sigma)$ as a proper coloring of $(G,\sigma)$ using colors from $\{\pm1,\pm2,\dots,\pm\frac{k}{2}\}$ if $k\equiv0~(\text{mod}~2)$, and ones from $\{0,\pm1,\pm2,\dots,\pm\frac{k-1}{2}\}$ if $k\equiv1~(\text{mod}~2)$.
A signed graph $(G,\sigma)$ is \emph{$k$-colorable} if it admits a $k$-coloring.
The \emph{chromatic number} of $(G,\sigma)$ is the minimum number $k$ such that $(G,\sigma)$ is $k$-colorable. We follow the approach of \cite{Raspaud_2014} 
to define list colorings of signed graphs. Given a signed graph $(G,\sigma)$, a \emph{list-assignment} of $(G,\sigma)$ is a function $L$ defined on $V(G)$ such that $\emptyset\neq L(v)\subseteq \mathbb{Z}$ for each $v\in V(G)$. An \emph{$L$-coloring} of $(G,\sigma)$ is a proper coloring $c$ of $(G,\sigma)$ such that $c(v)\in L(v)$ for each $v\in V(G)$.
A list-assignment $L$ is called a \emph{$k$-list-assignment} if $|L(v)|=k$ for each $v\in V(G)$.
We say $(G,\sigma)$ is \emph{$k$-choosable} if it admits an $L$-coloring for every $k$-list-assignment $L$.
The \emph{choice number} of $(G,\sigma)$ is the minimum number $k$ such that $(G,\sigma)$ is $k$-choosable.
Clearly, if a signed graph is $k$-choosable, then it is also $k$-colorable.

Let $(G,\sigma)$ be a signed graph, $L$ be a list assignment of $(G,\sigma)$, and $c$ be an $L$-coloring of $(G,\sigma)$. Let $X \subseteq V(G)$. We say $\sigma', L'$ and $c'$ are obtained from $\sigma, L$ and $c$ by a \emph{switch} at $X$ if
$$\sigma '(e)=
\begin{cases}
&-\sigma (e),\quad \textrm{if} ~e\in \partial(X), \\
&\sigma (e), \quad \textrm{if} ~e\in E(G)\setminus \partial(X),
\end{cases}
\quad\quad
L'(u)=
\begin{cases}
&\{-\alpha \colon\ \alpha\in L(u)\},\quad \textrm{if} ~u\in X, \\
&L(u), \quad \textrm{if} ~u\in V(G)\setminus X,
\end{cases}
$$
$$
c'(u)=
\begin{cases}
&-c(u),\quad \textrm{if} ~u\in X, \\
&c(u), \quad \textrm{if} ~u\in V(G)\setminus X.
\end{cases}
$$
Two signed graphs $(G,\sigma)$ and $(G,\sigma^*)$ are {\em equivalent} if they can be obtained from each other by a switch at some subset of $V(G)$. Let $\mathcal{G}(G,\sigma)=\{(G,\sigma_1)\colon\ (G,\sigma_1)$ is equivalent to $(G,\sigma)\}$.

\begin{proposition} \label{switch}
Let $(G,\sigma)$ be a signed graph, $L$ be a list-assignment of $G$ and $c$ be an $L$-coloring of $(G,\sigma)$.
If $\sigma', L'$ and $c'$ are obtained from $\sigma, L$ and $c$ by a switch at a subset of $V(G)$,
then $c'$ is an $L'$-coloring of $(G,\sigma')$. Furthermore, two equivalent signed graphs have the same chromatic number and the same choice number.
\end{proposition}

Let $G$ be a graph. By definition, $G$ and $(G,+)$ have the same chromatic number and the same choice number.
Hence, the following statement holds.
\begin{corollary} \label{swith_corollary}
If $(G,\sigma)\in \mathcal{G}(G,+)$, then $G$ and $(G,\sigma)$ have the same chromatic number and the same choice number.
\end{corollary}

This paper focusses on the choosability of signed planar graphs and generalizes the results of \cite{Fijavz_Juvan_Mohar_Skrekovski_2002, Lam_1999, Thomassen_1994, Thomassen_1995, Thomassen_2003, WeiFanWang_KWLih_2002} to signed graphs.
Section \ref{5c} proves that every signed planar graph is 5-choosable. Furthermore, there is a
signed planar graph $(G,\sigma)$  which is not 4-choosable, but $(G,+)$ is 4-choosable.
Section \ref{4c} proves for every $k \in \{3,4,5,6\}$ that every signed planar graph without $k$-circuits is 4-choosable.
Section \ref{3c} proves that every signed planar graph with neither 3-circuits nor 4-circuits is 3-choosable.
Furthermore,
there exists a signed planar graph $(G,\sigma)$ such that $G$ has girth 4 and $(G,\sigma)$ is not 3-choosable but $(G,+)$ is 3-choosable.

\section{5-choosability} \label{5c}

\begin{theorem} \label{5_choosable}
Every signed planar graph is 5-choosable.
\end{theorem}

We use the method described in \cite{Thomassen_1994} to prove following theorem which implies Theorem \ref{5_choosable}.
A plane graph $G$ is a {\em near triangulation} if the boundary of each bounded face of $G$ is a triangle.

\begin{theorem} \label{5_choosable_extend}
Let $(G,\sigma)$ be a signed graph, where $G$ is a near-triangulation. Let $C$ be the boundary of the unbounded face of $G$ and $C=[v_1\ldots v_p]$.
If $L$ is a list-assignment of $(G,\sigma)$ such that $L(v_1)=\{\alpha\}$, $L(v_2)=\{\beta\}$ and $\alpha \neq \beta  \sigma(v_1v_2)$,
and that $|L(v)|\geq 3$ for $v\in V(C)\setminus \{v_1,v_2\}$ and $|L(v)|\geq 5$ for $v\in V(G)\setminus V(C)$, then $(G,\sigma)$ has an $L$-coloring.
\end{theorem}

\begin{proof}
Let us prove Theorem \ref{5_choosable_extend} by induction on $|V(G)|$.

If $|V(G)|=3$, then $p=3$ and $G=C$. Choose a color from $L(v_3)\setminus \{\alpha  \sigma(v_1v_3), \beta  \sigma(v_2v_3)\}$ for $v_3$. So we proceed to the induction step.

If $C$ has a chord which divides $G$ into two graphs $G_1$ and $G_2$, then we choose the notation such that $G_1$ contains $v_1v_2$, and  we apply the induction hypothesis first to $G_1$ and then to $G_2$. Hence, we can assume that $C$ has no chord.

Let $v_1,u_1,u_2,\ldots,u_m,v_{p-1}$ be the neighbors of $v_p$ in cyclic order around $v_p$.
Since the boundary of each bounded face of $G$ is a triangle, $G$ contains the path $P\colon v_1u_1\ldots u_mv_{p-1}$.
Since $C$ has no chord, $P\cup (C-v_p)$ is a circuit $C'$. Let $\gamma_1$ and $\gamma_2$ be two distinct colors of $L(v_p)\setminus \{\alpha  \sigma(v_1v_p)\}$. Define $L'(x)=L(x)\setminus \{\gamma_1  \sigma(v_px), \gamma_2  \sigma(v_px)\}$ for $x\in \{u_1,\ldots,u_m\}$, and $L'(x)=L(x)$ for $x\in V(G)\setminus \{v_p,u_1,\ldots,u_m\}$.
Let $\sigma'$ be the restriction of $\sigma$ to $G-v_p$.
By the induction hypothesis, signed graph $(G-v_p,\sigma')$ has an $L'$-coloring. Let $c$ be the color vertex $v_{p-1}$ receives. We choose a color from $\{\gamma_1, \gamma_2\}\setminus \{c \sigma(v_{p-1}v_p)\}$ for $v_p$, giving an $L$-coloring of $(G,\sigma)$.
\end{proof}

\subsection*{non-4-choosable examples}
Voigt \cite{Voigt_1993, Voigt_1997} constructed two planar graphs which are not 4-choosable.
By Corollary \ref{swith_corollary} these two examples generate two group of signed planar graphs which are not 4-choosable.
We extend this result to signed graphs.

\begin{theorem}
There exists a signed planar graph $(G,\sigma)$ such that $(G,\sigma)$ is not 4-choosable but $G$ is 4-choosable.
\end{theorem}

\begin{proof}
We construct $(G,\sigma)$ as follows.
Take a copy $G_1$ of complete graph $K_4$ and embed it into Euclidean plane.
Insert a claw into each 3-face of $G_1$ and denote the resulting graph by $G_2$.
Once again, insert a claw into each 3-face of $G_2$ and denote by $G_3$ the resulting graph.
A vertex $v$ of $G_3$ is called an initial-vertex if $v\in V(G_1)$, a solid-vertex if $v\in V(G_2)\setminus V(G_1)$ and a hollow-vertex if $v\in V(G_3)\setminus V(G_2)$ (Figure \ref{Fig1} illustrates graph $G_3$).
A 3-face of $G_3$ is called a special 3-face if it contains an initial-vertex, a solid-vertex and a hollow-vertex.
Clearly, $G_3$ has 24 special 3-faces, say $T_1,\ldots,T_{24}$.

Let $H$ be the plane graph as shown in Figure \ref{Fig2}, which consists of a circuit $[xyz]$ and its interior.
For $i\in \{1,\dots,24\}$, replace $T_i$ by a copy $H_i$ of $H$ such that $x_i, y_i$ and $z_i$ are identified with the solid-vertex, hollow-vertex and initial-vertex of $T_i$, respectively.
Let $G$ be the resulting graph. Clearly, $G$ is planar.

\begin{figure}[hh]
  \centering
  \includegraphics[width=5cm]{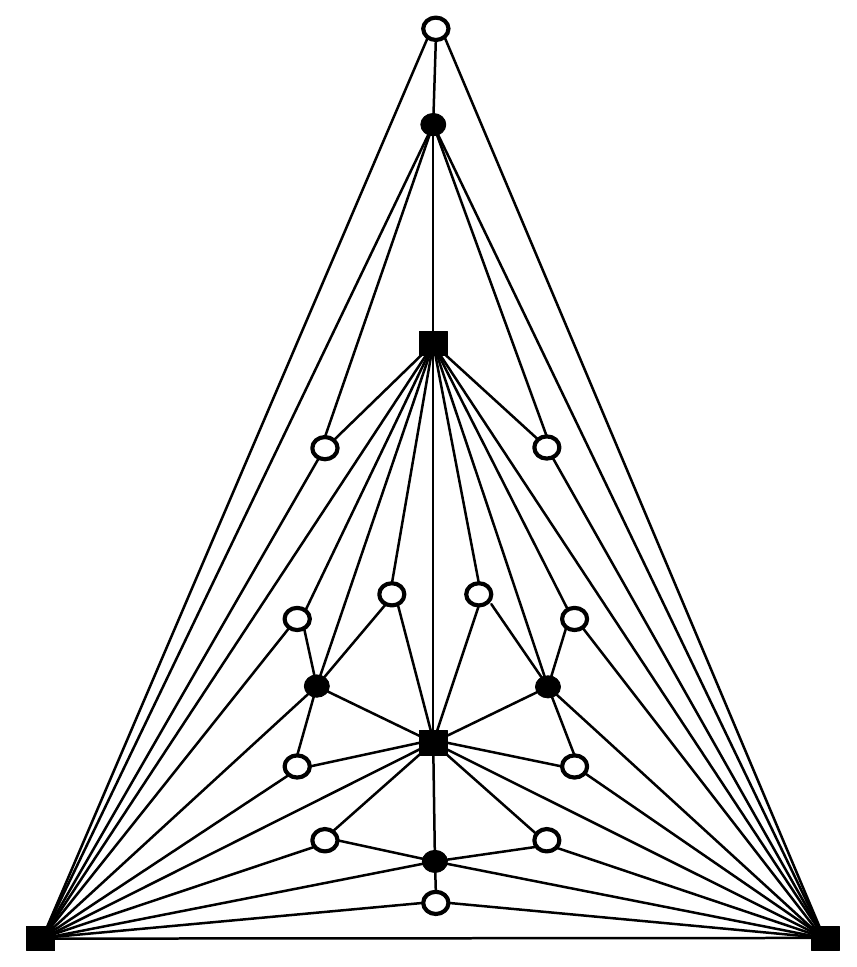}\\
  \caption{graph $G_3$}\label{Fig1}
\end{figure}

\begin{figure}[hh]
  \centering
  \includegraphics[width=6cm]{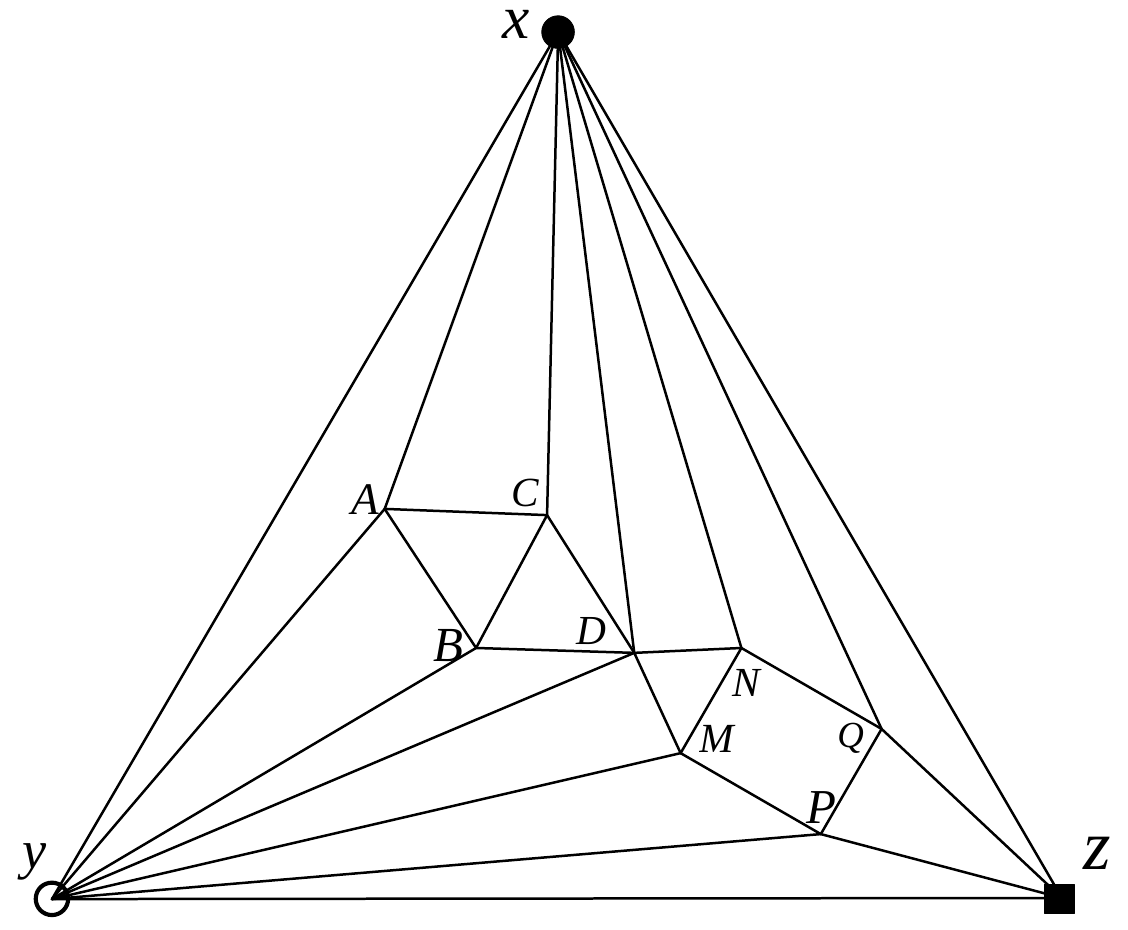}\\
  \caption{graph $H$}\label{Fig2}
\end{figure}

Define a signature $\sigma$ of $G$ as follows: $\sigma(P_iQ_i)=-1$ for $i \in \{1,\ldots,24\}$ and $\sigma(e)=1$ for $e\in E(G)\setminus \{P_iQ_i: i\in \{1,\ldots,24\}\}$.

Let $L$ be a 4-list-assignment of signed graph $(G,\sigma)$ defined as follows: $L(v)=\{1,2,3,4\}$ for $v\in V(G_3)$, and
$L(A_i)=\{1,2,6,7\}$, $L(B_i)=\{2,4,6,7\}$, $L(C_i)=\{1,4,6,7\}$, $L(D_i)=\{1,2,4,5\}$, $L(M_i)=\{2,5,6,-6\}$, $L(N_i)=\{1,5,6,-6\}$, $L(P_i)=\{2,3,6,-6\}$ and $L(Q_i)=\{1,3,6,-6\}$ for $i\in \{1,\ldots,24\}$.

We claim that signed graph $(G,\sigma)$ has no $L$-coloring. Suppose to the contrary that $\phi$ is an $L$-coloring of $(G,\sigma)$.
By the construction of $G_3$, precisely one of the special 3-faces of $G_3$ is assigned in $\phi$ color 1 to its solid-vertex, color 2 to its hollow-vertex and color 3 to its initial-vertex.
Without loss of generality, let $T_1$ be such a special 3-face. Let us consider $\phi$ in $H_1$.
Clearly, $\phi(x_1)=1, \phi(y_1)=2$ and $\phi(z_1)=3$. It follows that $\phi(D_1)\in \{4,5\}$.
Notice that the odd circuit $[A_1B_1C_1]$ is balanced and the even circuit $[M_1N_1Q_1P_1]$ is unbalanced, and thus both of them are not 2-choosable.
It follows that if $\phi(D_1)=4$, then $\phi$ is not proper in $[A_1B_1C_1]$,
and that if $\phi(D_1)=5$, then $\phi$ is not proper in $[M_1N_1Q_1P_1]$.
Therefore, $(G,\sigma)$ has no $L$-coloring and thus is not 4-choosable.

Let $L'$ be any 4-list-assignment of $G$.
By the construction, it is not hard to see that $G_3$ is 4-choosable.
Let $c$ be an $L'$-coloring of $G_3$.
Clearly, for $i\in \{1,\ldots,24\}$, each of vertices $x_i, y_i$ and $z_i$ receives a color in $c$. Let $\alpha$ and $\beta$ be two distinct colors from $L(D_i)\setminus \{c(x_i), c(y_i)\}$.
Choose a color from $L(C_i)\setminus \{\alpha, \beta, c(x_i)\}$ for $C_i$, and then vertices $A_i, B_i$ and $D_i$ can be list-colored by $L'$ in turn.
Since circuit $[M_iN_iQ_iP_i]$ is 2-choosable, it follows that vertices $M_i, N_i, P_i$ and $Q_i$ can also be list-colored by $L'$.
Therefore, $c$ can be extended to an $L'$-coloring of $G$.
This completes the proof that $G$ is 4-choosable.
\end{proof}

\section{4-choosability} \label{4c}

A graph $G$ is \emph{$d$-degenerate} if every subgraph $H$ of $G$ has a vertex of degree at most $d$ in $H$.
It is known that every $(d-1)$-degenerate graph is $d$-choosable. This proposition can be extended for signed graphs.

\begin{theorem} \label{degenerate_choosability}
Let $(G,\sigma)$ be a signed graph. If $G$ is $(d-1)$-degenerate, then $(G,\sigma)$ is $d$-choosable.
\end{theorem}

\begin{proof}
(induction on $|V(G)|$) Let $L$ be any $d$-list-assignment of $G$. The proof is trivial if $|V(G)|=1$. For $|V(G)|\geq 2$, since $G$ is $(d-1)$-degenerate, $G$ has a vertex $v$ of degree at most $d-1$ and moreover, graph $G-v$ is $(d-1)$-degenerate. Let $\sigma'$ and $L'$ be the restriction of $\sigma$ and $L$ to $G-v$, respectively. By applying the induction hypothesis to $(G-v,\sigma')$, we conclude that $(G-v,\sigma')$ is $d$-choosable and thus has an $L'$-coloring $\phi$. Since $v$ has degree at most $d-1$, we can choose a color $\alpha$ for $v$ such that $\alpha\in L(v)\setminus \{\phi(u)\sigma(uv)\colon\ uv\in E(G)\}$.
We complete an $L$-coloring of $(G,\sigma)$ with $\phi$ and $\alpha$.
\end{proof}

It is an easy consequence of Euler's formula that every triangle-free planar graph contains a vertex of degree at most 3. Therefore, the following statement is true:
\begin{lemma} \label{3-degenerate3}
Planar graphs without 3-circuits are 3-degenerate.
\end{lemma}

Moreover, we will use two more lemmas.
\begin{lemma} [\cite {WeiFanWang_KWLih_2002}] \label{3-degenerate4}

Planar graphs without 5-circuits are 3-degenerate.
\end{lemma}

\begin{lemma} [\cite{Fijavz_Juvan_Mohar_Skrekovski_2002}] \label{3-degenerate5}
Planar graph without 6-circuits are 3-degenerate.
\end{lemma}

\begin{theorem} \label{4_choosable}
Let $(G,\sigma)$ be a signed planar graph. For all $k\in \{3,4,5,6\}$, if $G$ has no $k$-circuit, then $(G,\sigma)$ is 4-choosable.
\end{theorem}

\begin{proof}
For $k\in \{3,5,6\}$ we deduce the statement from Theorem \ref{degenerate_choosability}, together with Lemmas \ref{3-degenerate3}, \ref{3-degenerate4} and \ref{3-degenerate5}, respecvitely.
It remains to prove Theorem \ref{4_choosable} for the case $k=4$.

Suppose to the contrary that the statement is not true.
Let $(G,\sigma)$ be a counterexample of smallest order, and $L$ be a 4-list-assignment of $(G,\sigma)$ such that $(G,\sigma)$ has no $L$-coloring. Clearly, $G$ is connected by the minimality of $(G,\sigma)$.

\begin{claim} \label{minimal degree}
$\delta(G)\geq 4$.
\end{claim}
Let $u$ be a vertex of $G$ of minimal degree. Suppose to the contrary that $d(u)<4$.
Let $\sigma'$ and $L'$ be the restriction of $\sigma$ and $L$ to $G-u$, respectively.
By the minimality of $(G,\sigma)$, the signed graph $(G-u,\sigma')$ has an $L'$-coloring $c$.
Since every neighbor of $u$ forbids one color for $u$ no matter what the signature of the edge between them is, $L(u)$ still has a color left for coloring $u$.
Therefore, $c$ can be extended to an $L$-coloring of $(G,\sigma)$, a contradiction.

\begin{claim} \label{5face}
$G$ has no 6-circuit $C$ such that $C=[u_0\ldots u_5]$ and $u_0u_2\in E(G)$, and $d(u_0)\leq 5$ and all other vertices of $C$ are of degree 4.
\end{claim}
Suppose to the contrary that $G$ has such 6-circuit $C$.
Since $G$ has no 4-circuit, $u_0u_2$ is the only chord of $C$.
There always exists a subset $X$ of $V(C)$ such that all of the edges $u_0u_2, u_1u_2$ and $u_2u_3$ are positive after a switch at $X$.
Let $\sigma'$ and $L'$ be obtained from $\sigma$ and $L$ by a switch at $X$, respectively.
Proposition \ref{switch} implies that signed graph $(G,\sigma')$ has no $L'$-coloring.
Hence, $(G,\sigma')$ is also a minimal counterexample.
Let $\sigma_1$ and $L_1$ be the restriction of $\sigma'$ and $L'$ to $G-V(C)$, respectively.
It follows that $(G-V(C),\sigma_1)$ has an $L_1$-coloring $\phi$.

We obtain a contradiction by further extending $\phi$ to an $L'$-coloring of $(G,\sigma')$ as follows.
By the condition on the vertex degrees of $C$, there exists a list-assignment $L_2$ of $G[V(C)]$ such that $L_2(u)\subseteq L'(u)\setminus \{\phi(v)\sigma'(uv)\colon\ uv\in E(G)~and~v\notin V(C)\}$ for $u\in V(C)$, and $|L_2(u_2)|=3$ and $|L_2(u)|= 2$ for $u\in V(C)\setminus \{u_2\}$.
Let $L_2(u_2)=\{\alpha,\beta,\gamma\}$.
Suppose that $L_2(u_2)$ has a color, say $\alpha$, not appear in at least two of lists $L_2(u_0), L_2(u_1)$ and $L_2(u_3)$.
We color $u_2$ with $\alpha$, and then all other vertices of $C$ can be list-colored by $L_2$ in some order.
For example, if $\alpha$ does not appear in $L_2(u_0)$ and $L_2(u_1)$, then we color $V(C)$ in the order
$u_2,u_3,u_4, u_5,u_0,u_1$.
Hence, we may assume that $L_2(u_0)=\{\alpha,\gamma\}$, $L_2(u_1)=\{\alpha,\beta\}$ and $L_2(u_3)=\{\beta,\gamma\}$.
If $\beta \neq \gamma \sigma'(u_0u_1)$ , then color $u_0$ with $\gamma$, $u_1$ with $\beta$, and $u_2$ with $\alpha$, and the remaining vertices of $C$ can be list-colored by $L_2$ in the order $u_5,u_4,u_3$.
Hence, we may assume $\beta = \gamma \sigma'(u_0u_1)$. It follows that $\sigma'(u_0u_1)=-1$ and $\beta=-\gamma\neq0$.
If $\alpha \neq 0$, then color both $u_0$ and $u_1$ with $\alpha$, and the remaining vertices of $C$ can be list-colored by $L_2$ in the order $u_5,u_4,u_3,u_2$.
Hence, we may assume $\alpha = 0$.
Now color 0 is included in list $L_2(u_0)$ but no in list $L_2(u_3)$.
Thus there exists an integer $i$ in set $\{3,4,5\}$ such that $0\in L_2(u_{i+1})$ and $0\notin L_2(u_i)$ (index is added modular 6).
We color $u_{i+1}$ with color 0, and then the remaining vertices of $C$ can be list-colored by $L_2$ in cyclic order on $C$ ending at $u_i$.

\begin{claim} \label{6-vertex}
$G$ has no 10-circuit $C$ such that $C=[u_0\ldots u_9]$ and $u_0u_8, u_2u_6, u_2u_7 \in E(G)$, and vertex $u_2$ has degree 6 and all other vertices of $C$ have degree 4.
\end{claim}
Suppose to the contrary that $G$ has such a 10-circuit $C$.
Let $\sigma'$ and $L'$ be the restriction of $\sigma$ and $L$ to graph $G-V(C)$, respectively.
By the minimality of $(G,\sigma)$, signed graph $(G-V(C),\sigma')$ has an $L'$-coloring $\phi$.
A contradiction is obtained by further extending $\phi$ to an $L$-coloring of $(G,\sigma)$ as follows.
We shall list-color the vertices of $C$ by $L$ in the cyclic order $u_0,u_1,\ldots, u_9$.
For $i\in \{0,\ldots,9\}$, let $F_i=\{\phi(v)\sigma(u_iv)\colon\ u_iv\in E(G)~and~v\notin V(C)\}$.
Clearly, $F_i$ is the set of forbidden colors by the neighbors of $u_i$ not on $C$ to be assigned to vertex $u_i$.
Since $d(u_0)=d(u_9)=4$ and moreover, if there is any other chord of $C$ then the list $F_i$ will not become longer, it follows that $|F_0|\leq1$ and $|F_9|\leq2$.
Hence, we can let $\alpha$ and $\beta$ be two distinct colors from $L(u_9)\setminus F_9$, and let $\gamma\in L(u_0)\setminus (F_0\cup \{\alpha\sigma(u_0u_9),\beta\sigma(u_0u_9)\})$.
Color vertex $u_0$ with $\gamma$.
For $i\in\{1,\ldots,8\}$, vertex $u_i$ has at most 3 neighbors colored before $u_i$ in this color-assigning process and thus,
$L(u_i)$ still has a color available for $u_i$.
Denote by $\zeta$ the color vertex $u_8$ receives.
We complete the extending of $\phi$ by assigning a color from $\{\alpha,\beta\}\setminus \{\zeta\sigma(u_8u_9)\}$ to $u_9$.

\subsection*{Discharging}
Consider an embedding of $G$ into the Euclidean plane.
Let $G$ denote the resulting plane graph.
We say two faces are \emph{adjacent} if they share an edge.
Two adjacent faces are \emph{normally adjacent} if they share an edge $xy$ and no vertex other than $x$ and $y$.
Since $G$ is a simple graph, the boundary of every 3-face or 5-face is a circuit.
Since $G$ has no 4-circuits, we can deduce that if a 3-face and a 5-face are adjacent, then they are normally adjacent.
A vertex is \emph{bad} if it is of degree 4 and incident with two nonadjacent 3-faces.
A \emph{bad 3-face} is a 3-face containing three bad vertices.
A 5-face $f$ is \emph{magic} if it is adjacent to five 3-faces, and if all the vertices of these six faces have degree 4 except one vertex of $f$.

We shall obtain a contradiction by applying discharging method.
Let $V=V(G)$, $E=E(G)$, and $F$ be the set of faces of $G$.
Denote by $d(f)$ the size of a face $f$ of $G$.
Give initial charge $ch(x)$ to each element $x$ of $V\cup F$, where $ch(v)=3d(v)-10$ for $v\in V$, and $ch(f)=2d(f)-10$ for $f\in F$.
Discharge the elements of $V\cup F$ according to the following rules:

\begin{enumerate}[R1.]
  \itemsep=0cm
  \item Every vertex $u$ sends each incident 3-face charge 1 if $u$ is a bad vertex, and charge 2 otherwise.
  \item Every 5-vertex sends $\frac{1}{3}$ to each incident 5-face.
  \item Every 6-vertex sends each incident 5-face $f$ charge 1 if $f$ is magic, charge $\frac{2}{3}$ if $f$ is not magic but contains four 4-vertices, charge $\frac{1}{3}$ if $f$ contains at most three 4-vertices.
  \item Every $7^+$-vertex sends 1 to each incident 5-face.
  \item Every 3-face sends $\frac{1}{3}$ to each adjacent 5-face if this 3-face contains at most one bad vertex.
  \item Every $5^+$-face sends $\frac{k}{3}$ to each adjacent bad 3-face, where $k$ is the number of common edges between them.
\end{enumerate}

Let $ch^*(x)$ denote the final charge of each element $x$ of $V\cup F$ when the discharging process is over.
On one hand, by Euler's formula we deduce $\sum\limits_{x\in V\cup F}ch(x)=-20.$
Since the sum of charge over all elements of $V\cup F$ is unchanged, we have $\sum\limits_{x\in V\cup F}ch^*(x)=-20.$ On the other hand, we show that $ch^*(x)\geq 0$ for $x\in V\cup F$. Hence, this obvious contradiction completes the proof of Theorem \ref{4_choosable}.

It remains to show that $ch^*(x)\geq 0$ for $x\in V\cup F$.

\begin{claim} \label{claim_vertex}
If $v\in V$, then $ch^*(v)\geq0$.
\end{claim}

Let $p$ be the number of 3-faces that contains $v$.
Since $G$ has no 4-circuit, $p\leq \lfloor \frac{d(v)}{2}\rfloor$.
Moreover, $d(v)\geq4$ by Claim \ref{minimal degree}.

Suppose $d(v)=4$. We have $p\leq 2$. If $p=2$, then $v$ is a bad vertex and thus we have $ch^*(v)= 3d(v)-10-p=0$ by R1; otherwise, we have $ch^*(v)= 3d(v)-10-2p\geq 0$ by R1 again.

If $d(v)=5$, then $p\leq 2$ and thus by R1 and R2, we have $ch^*(v)\geq 3d(v)-10-2p-\frac{1}{3}(5-p)\geq0$.

Suppose that $d(v)=6$. Thus $p\leq 3$. By R1 and R3, if $p\leq 2$ then we have $ch^*(v)\geq 3d(v)-10-2p-(6-p)\geq 0$, and if $v$ is incident with no magic 5-face then we have $ch^*(v)\geq 3d(v)-10-2p-\frac{2}{3}(6-p)\geq 0$. Hence, we may assume that $p=3$ and that $v$ is incident with a magic 5-face $f$. For any other $5^+$-face $f'$ containing $v$ than $f$, Claim \ref{6-vertex} implies that if $f'$ has size 5 then it contains at most three 4-vertices, and thus $v$ sends at most $\frac{1}{3}$ to $f'$ by R3. Hence, we have $ch^*(v)\geq 3d(v)-10-2\times 3-1-\frac{1}{3}\times 2> 0$.

It remains to suppose $d(v)\geq7$. By R1 and R4, we have $ch^*(v)\geq 3d(v)-10-2p-(d(v)-p)\geq 2d(v)-10-\lfloor \frac{d(v)}{2}\rfloor >0$.
\begin{claim}
If $f\in F$, then $ch^*(f)\geq0$.
\end{claim}
Suppose $d(f)=3.$ Recall that in this case the boundary of $f$ is a circuit. We have $ch^*(f)\geq 2d(f)-10+2+2+1-3\times \frac{1}{3}=0$ by R1 and R5 when $f$ has at most one bad vertex, and $ch^*(f)\geq 2d(f)-10+2+1+1=0$ by R1 when $f$ has precisely two bad vertices. It remains to assume that $f$ has precisely three bad vertices, that is, $f$ is a bad 3-face. In this case, $f$ receives charge 1 in total from adjacent faces by R6, and charge 3 in total from incident vertices by R1. Hence, we have $ch^*(f)\geq 2d(f)-10+1+3=0$.

Suppose $d(f)=5$. Recall in this case that the boundary of $f$ is a circuit and that if $f$ is adjacent to a 3-face then they are normally adjacent.
Let $q$ be the number of bad 3-faces adjacent to $f$.
Clearly, $f$ sends charge only to adjacent bad 3-faces by R6, and possibly receives charge from incident $5^+$-vertices and adjacent 3-faces by rules from R2 to R5.
Hence, we have $ch^*(f)\geq 2d(f)-10=0$ when $q=0$.
Claim \ref{5face} implies that $q\leq3$ and that $f$ contains a $5^+$-vertex $u$, which sends at least $\frac{1}{3}$ to $f$.
Hence, we have $ch^*(f)\geq 2d(f)-10-\frac{1}{3}+\frac{1}{3}=0$ when $q=1$.
First suppose $q=2$. If $f$ has a $5^+$-vertices different from $u$, then we are done by $ch^*(f)\geq 2d(f)-10-2\times \frac{1}{3}+2\times \frac{1}{3}=0$. Hence, we may assume that $f$ contains four 4-vertices. It follows that if $d(u)\geq 6$, then $f$ receives at least $\frac{2}{3}$ from $v$ by R3 or R4 and thus we are done. Hence, we may assume that $d(u)=5$.
Through the drawing of 3-faces adjacent to $f$, we can assume $u$ is incident with a 3-face $[uvw]$ that is adjacent to $f$ on edge $uv$.
Claim \ref{5face} implies that $d(w)\geq 5$. Hence, $f$ receives $\frac{1}{3}$ from face $[uvw]$ by R5, and thus we are done.
Let us next suppose $q=3$. We may assume $f=[uv'w'x'y']$ such that $v'w'$, $w'x'$ and $x'y'$ are the three common edges between $f$ and bad 3-faces.
Since both vertices $v'$ and $y'$ are bad, edges $uv'$ and $uy'$ are contained in 3-faces $[uv't']$ and $[uy'z']$, respectively.
If $d(u)=5,$ then Claim \ref{5face} implies that $d(t'), d(z')\geq 5$, and thus $f$ receives $\frac{1}{3}$ from each of faces $[uv't']$ and $[uy'z']$ by R5, we are done.
If $d(u)\geq 7$, then $f$ receives 1 from $u$ and thus we are done.
Hence, we may assume that $d(u)=6$.
If both $t'$ and $z'$ has degree 4, that is, $f$ is a magic 5-face, then $f$ receives 1 from $u$ by R3;
otherwise, $f$ receives $\frac{2}{3}$ from $u$ and $\frac{1}{3}$ from at least one of faces $[uv't']$ and $[uy'z']$ by R3 again.
We are done in both cases.

It remains to suppose $d(f)\geq 6$. Remind that $f$ has no charge moving in or out except that it sends $\frac{1}{3}d(f)$ in total to adjacent bad 3-faces by R6.
Hence, we have $ch^*(f)\geq 2\times d(f)-10-\frac{1}{3}d(f)\geq 0$.

The proof of Theorem \ref{4_choosable} is completed.
\end{proof}

\section{3-choosability} \label{3c}
In 1995, Thomassen \cite{Thomassen_1995} proved that every planar graph of girth at least 5 is 3-choosable. And then in 2003, he \cite{Thomassen_2003} gave a shorter proof of this result.
We find out that the argument used in \cite{Thomassen_2003} also works for signed graphs.
Hence the following statement is true.

\begin{theorem} \label{3_choosable}
Every signed planar graph with neither 3-circuit nor 4-circuit is 3-choosable.
\end{theorem}

For the sake of completeness, the proof is given in the appendix.

\begin{theorem}
There exists a signed planar graph $(G,\sigma)$ such that $G$ has girth 4 and $(G,\sigma)$ is not 3-choosable but $G$ is 3-choosable.
\end{theorem}

\begin{proof}
Let $T$ be a plane graph consisting of two circuits $[ABCD]$ and $[MNPQ]$ of length 4 and four other edges $AM, BN, CP$ and $DQ$, as shown in Figure \ref{fig3}. Take nine copies $T_0,\ldots,T_8$ of $T$, and identify $A_0,\ldots,A_8$ into a vertex $A'$ and $C_0,\ldots,C_8$ into a vertex $C'$. Let $G$ be the resulting graph. Clearly, $G$ is planar and has girth 4.
\begin{figure}[h]
  \centering
  \includegraphics[width=4cm]{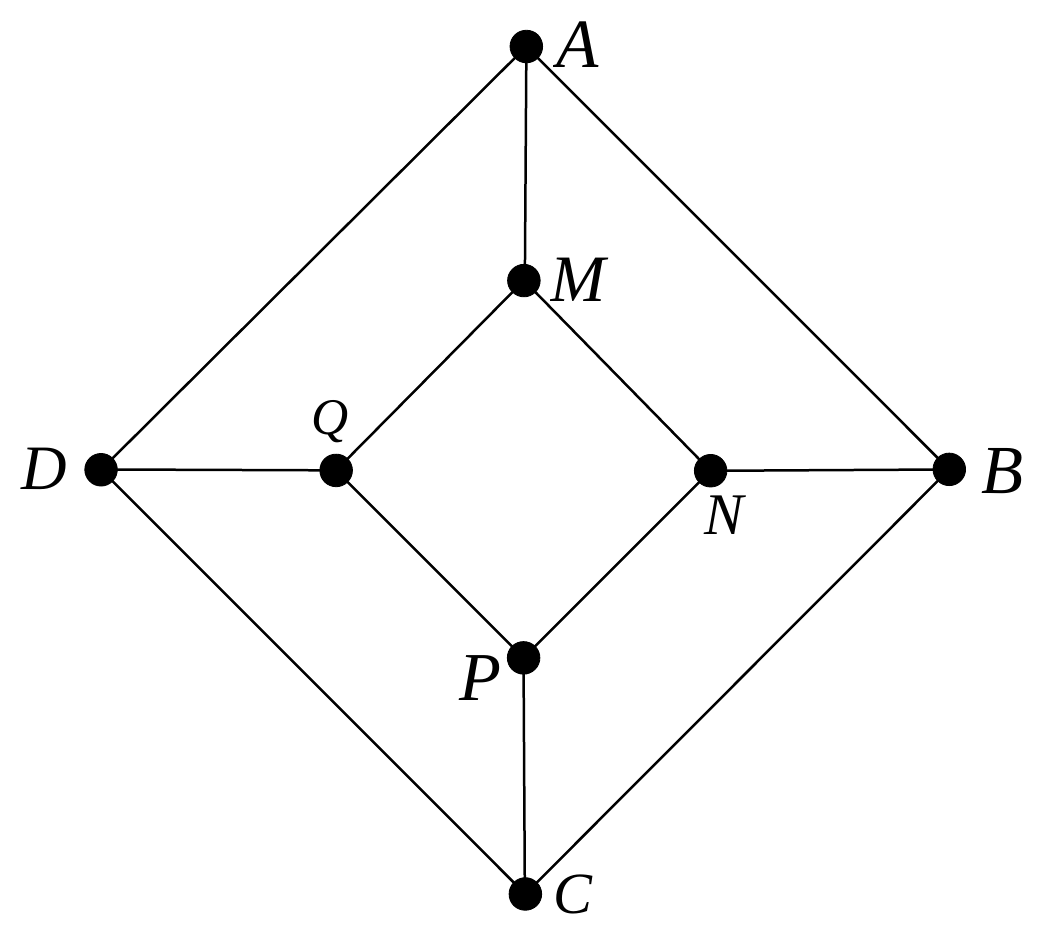}\\
  \caption{graph $T$}\label{fig3}
\end{figure}

Define a signature $\sigma$ of $G$ as: $\sigma(e)=-1$ for $e\in \{M_iN_i: i\in \{0,\ldots,8\}\}$, and $\sigma(e)=1$ for $e\in E(G)\setminus \{M_iN_i: i\in \{0,\ldots,8\}\}$.

For $i\in \{0,1,2\}$, let $a_i=i$ and $b_i=i+3.$ Define a 3-list-assignment $L$ of $G$ as follows: $L(A')=\{a_1, a_2, a_3\}, L(C')=\{b_1, b_2, b_3\}$; for $i,j\in \{0,1,2\}$, let $L(B_{3i+j})=L(D_{3i+j})=\{a_i,b_j,6\}, L(N_{3i+j})=L(Q_{3i+j})=\{6,7,-7\}, L(M_{3i+j})=\{a_i,7,-7\}$ and $L(P_{3i+j})=\{b_j,7,-7\}$.

We claim that signed graph $(G,\sigma)$ has no $L$-coloring. Suppose to the contrary that $c$ is an $L$-coloring of $(G,\sigma)$.
Let $c(A')=a_p$ and $c(C')=b_q$. Consider subgraph $T_{3p+q}$. It follows that $c(B_{3p+q})=c(D_{3p+q})=6$.
Furthermore, the circuit $[M_{3p+q}N_{3p+q}P_{3p+q}Q_{3p+q}]$ is unbalanced and thus not 2-choosable.
Hence, $T_{3p+q}$ is not properly colored in $c$, a contradiction. This proves that $(G,\sigma)$ has no $L$-coloring and therefore, $(G,\sigma)$ is not 3-choosable.

We claim that graph $G$ is 3-choosable. For any 3-list-assignment of $G$, choose any color for vertices $A'$ and $C'$ from their color lists, respectively.
Consider each subgraph $T_i~(i\in \{0,\ldots,8\})$. Both vertices $B_i$ and $D_i$ can be list colored. The 2-choosability of circuit $[M_iN_iP_iQ_i]$ yields a list coloring of $T_i$ and hence a list coloring of $G$. This proves that $G$ is 3-choosable.
\end{proof}

\section{Appendix}

\begin{theorem}\label{Thomassen}
Let $(G,\sigma)$ be a signed plane graph of girth at least 5, and $D$ be the outer face boundary of $G$. Let $P$ be a path or circuit of $G$ such that $|V(P)|\leq 6$ and $V(P)\subseteq V(D)$,
and $\sigma_p$ be the restriction of $\sigma$ to $P$.
Assume that $(P,\sigma_p)$ has a 3-coloring $c$.
Let $L$ be a list-assignment of $G$ such that $L(v)=\{c(v)\}$ if $v\in V(P)$, $|L(v)|\geq 2$ if $v\in V(D)\setminus V(P)$, and $|L(v)|\geq 3$ if $v\in V(G)\setminus V(D)$.
Assume furthermore that there is no edge joining vertices whose lists have at most two colors except for the edges in $P$. Then $c$ can be extended to an $L$-coloring of $(G,\sigma)$.
\end{theorem}

\begin{proof}
We prove Theorem \ref{Thomassen} by induction on the number of vertices. We assume that $(G,\sigma)$ is a smallest counterexample and shall get a contradiction.

\begin{claim}\label{2-connected}
$G$ is 2-connected and hence, $D$ is a circuit.
\end{claim}

We may assume that $G$ is connected, since otherwise we apply the induction hypothesis to every connected component of $G$. Similarly, $G$ has no cutvertex in $P$.
Moreover, $G$ has no cutvertex at all.
Suppose to the contrary that $u$ is a cutvertex contained in an endblock $B$ disjoint from $P$. We first apply the induction hypothesis to $G-(B-u)$. If $B$ has vertices with only two available colors joined to $u$, then we color each such vertex. These colored vertices of $B$ together with the edges joining them to $u$ divide $B$ into parts each of which has at most three colored vertices inducing a path. Now we apply the induction hypothesis to each of those parts. This contradiction proves Claim \ref{2-connected}.

\begin{claim}\label{P_nonchord}
For $e\in E(P)$, $e$ is not a chord of $D$.
\end{claim}

If some edge $e$ of $P$ is a chord of $G$, then $e$ divides $G$ into two parts, and we apply the induction hypothesis to each of those two parts. This contradiction proves Claim \ref{P_nonchord}.

By Claims \ref{2-connected} and \ref{P_nonchord}, we may choose the notion such that $D=[v_1\ldots v_k]$ and $P=v_1\ldots v_q$.

Let $X$ be a set of colored vertices of $G$.
To save writing we just say ``delete the product colors of $X$ from $G$'' instead of
``for $v\in V(G)\setminus X$, delete all of the colors in $\{c(u)\sigma(uv)\colon\ u\in X~and~uv\in E(G)\}$ from the list of $v$''.

\begin{claim}\label {path}
$P$ is a path, and $q+3\leq k$.
\end{claim}

If $P=D$, then we delete any vertex from $D$, and delete the product color of that vertex from $G$. If $P\neq D$ and $k< q+3$, then we color the vertices of $D$ not in $P$, we delete them together with their product colors from $G$.

Now we apply the induction hypothesis to the resulting graph $G'$, if possible.
As $G$ has grith at least 5, the vertices with precisely two available colors are independent.
For the same reason, such a vertex cannot be joined to two vertices of $P$.
However, such a vertex may be joined to precisely one vertex of $P$.
We then color it. Now the colored vertices of $G'$ divide $G'$ into parts each of which has at most 6 precolored vertices inducing a path. We then apply induction hypothesis to each of those parts. This contradiction proves Claim \ref {path}.

\begin{claim}\label {chord}
$D$ has no chord.
\end{claim}
Suppose to the contrary that $xy$ is a chord of $D$. Then $xy$ divides $G$ into two graphs $G_1, G_2$, say. We may choose the notation such that $G_2$ has no more vertices of $P$ than $G_1$ has, and subject to that condition, $|V(G_2)|$ is minimum. We apply the induction hypothesis first to $G_1$. In particular, $x$ and $y$ receive a color. The minimality of $G_2$ implies that the outer cycle of $G_2$ is chordless. So $G_2$ has at most two vertices which have only two available colors and which are joined to one of $x$ and $y$. We color any such vertex, and then we apply the induction hypothesis to $G_2$. This contradiction proves Claim \ref {chord}.

\begin{claim}\label{2-path}
$G$ has no path of the form $v_iuv_j$ where $u$ lies inside $D$, except possibly when $q=6$ and the path is of the form $v_4uv_7$ or $v_3uv_k$. In particular, $u$ has only two neighbors on $D$.
\end{claim}

We define $G_1$ and $G_2$ as in the proof of Claim \ref{chord}.
We apply the induction hypothesis first to $G_1$.
Although $u$ may be joined to several vertices with only two available colors, the minimality of $G_2$ implies that no such vertex is in $G_2-\{u,v_i,v_j\}$. There may be one or two vertices in $G_2-\{u,v_i,v_j\}$ that have only two available colors and which are joined to one of $v_i$ and $v_j$. We color any such vertex, and then at most six vertices of $G_2$ are colored. If possible, we apply the induction hypothesis to $G_2$. This is possible unless the coloring of $G_1$ is not valid in $G_2$. This happens only if $P$ has a vertex in $G_2$ joined to one of $v_i$ and $v_j$. This happens only if we have one of the two exceptional cases described in Claim \ref{2-path}.

\begin{claim}\label {3-path}
$G$ has no path of the form $v_iuwv_j$ such that $u$ and $w$ lie inside $D$, and $|L(v_i)|=2$.
Also, $G$ has no path $v_iuwv_j$ such that $u$ and $w$ lie inside $D$, $|L(v_i)|=3$, and $j\in \{1,q\}$.
\end{claim}

Repeating the arguments in Claims \ref{chord} and \ref{2-path}, we can easily get Claim \ref{3-path}.

\begin{claim}\label{empty cycle}
If $C$ is a circuit of $G$ distinct from $D$ and of length at most 6, then the interior of $C$ is empty.
\end{claim}

Otherwise, we can apply the induction hypothesis first to $C$ and its exterior and then to $C$ and its interior. This contradiction proves Claim \ref{empty cycle}.


If $|L(v_{q+2})|\geq3$, then we complete the proof by deleting $v_q$ and its product color from $G$, and apply the induction hypothesis to $G-v_q$ and obtain thereby a contradiction. So we assume $|L(v_{q+2})|\leq2$. By Claim \ref{path}, $|L(v_{q+2})|=2$ and thus $|L(v_{q+3})|\geq3$. If $|L(v_{q+4})|\geq3$, then we first color $v_{q+2}$ and $v_{q+1}$, then we delete them and their product colors from $G$. We obtain a contradiction by applying the induction hypothesis to the resulting graph. By Claims \ref{chord} and \ref{2-path} this is possible unless $q=6$ and $G$ has a vertex $u$ inside $D$ joined to both $v_4$ and $v_7$. In this case we color $u$ and delete both $v_5$ and $v_6$ before we apply the induction hypothesis.
Hence, we may assume that $|L(v_{q+4})|\leq 2$.

We give $v_{q+3}$ a color not in $\{ \alpha \sigma(v_{q+3}v_{q+4})\colon\ \alpha \in L(v_{q+4})\}$ and then color $v_{q+2}$ and $v_{q+1}$, and finally we delete $v_i$ and the product color of $v_i$ from $G$ for $i\in \{q+1,q+2,q+3\}$.
We obtain a contradiction by applying the induction hypothesis to the resulting graph.
If $q=6$ and $G$ has a vertex $u$ inside $D$ joined to $v_4$ and $v_7$, then, as above, we color $u$ and delete $v_5$ and $v_6$ before we use induction.
If $q=6, q+3=k$, and $G$ has a vertex $u'$ inside $D$ joined to $v_3$ and $v_k$, then we also color $u'$ and delete $v_1$ and $v_2$ before we use induction.
Finally, there may be a path $v_{q+1}wzv_{q+3}$ where $w$ and $z$ lies inside $D$.
By Claim \ref{empty cycle}, this path is unique.
We color $w$ and $z$ and delete them together with their product colors from $G$ before we use induction.
Note that $u$ and $u'$ may also exist in this case.
If there are vertices joined to two colored vertices, then we also color these vertices before we use induction.

The colored vertices divide $G$ into parts, and we shall show that each part satisfies the induction hypothesis.
By second statement of Claim \ref{3-path}, there are at most six precolored vertices in each part, and they induce a path.
Claim \ref{2-path} and the first statement of Claim \ref{3-path} imply that
there is no vertex with precisely two available colors on $D$ which is joined to a vertex inside $D$ whose list has only two available colors after the additional coloring.
Since $G$ has girth at least 5 and by Claim \ref{empty cycle},
there is no other possibility for two adjacent vertices $z$ and $z'$ to have only two available colors in their lists, as both $z$ and $z'$ must be adjacent to a vertex that has been colored and deleted.

This contradiction completes the proof.
\end{proof}
\end{document}